\documentclass[12pt]{amsart}
\usepackage{amsmath,amsthm,amsfonts,amssymb,mathrsfs}
\usepackage{color}

\usepackage{amssymb,cite}
\usepackage[colorlinks,plainpages,citecolor=magenta, linkcolor=blue, backref]{hyperref}%,bookmarksnumbered,

\usepackage{hyperref}
\date{\today}

\usepackage{hyperref}
\input{xy}
 \xyoption{all}
 \xyoption{arc}

\newtheorem{theorem}{Theorem}[section]

\newtheorem{proposition}[theorem]{Proposition}%[section]
\newtheorem{corollary}[theorem]{Corollary}%[section]

\newtheorem{lemma}[theorem]{Lemma}%[section]
\theoremstyle{definition}
\newtheorem{definition}[theorem]{Definition}
\newtheorem{example}[theorem]{Example}%[section]
\newtheorem{remark}[theorem]{Remark}%[section]

\begin{document}

\title[On the semigroup $\boldsymbol{B}_{\omega}^{\mathscr{F}}$ which is generated by the family $\mathscr{F}$ of atomic subsets of $\omega$]{On the semigroup $\boldsymbol{B}_{\omega}^{\mathscr{F}}$ which is generated by the family $\mathscr{F}$ of atomic subsets of $\omega$}
\author{Oleg Gutik and Oleksandra Lysetska}
\address{Faculty of Mechanics and Mathematics,
Lviv University, Universytetska 1, Lviv, 79000, Ukraine}
\email{oleg.gutik@lnu.edu.ua, o.yu.sobol@gmail.com}

\keywords{Semitopological semigroup, topological semigroup, bicyclic monoid, inverse semigroup, feebly compact, compact, Brandt $\omega$-extension, closure.}

\subjclass[2020]{Primary 22A15, 20A15,  Secondary 54D10, 54D30, 54H12}

\begin{abstract}
We study the semigroup $\boldsymbol{B}_{\omega}^{\mathscr{F}}$, which is introduced in [O. Gutik and M. Mykhalenych, \emph{On some generalization of the bicyclic monoid}, Visnyk Lviv. Univ. Ser. Mech.-Mat. \textbf{90} (2020), 5--19], in the case when the family $\mathscr{F}$ of subsets of cardinality $\leqslant 1$ in $\omega$. We show that $\boldsymbol{B}_{\omega}^{\mathscr{F}}$ is isomorphic to the  subsemigroup $\mathscr{B}_{\omega}^{\Rsh}(\boldsymbol{F}_{\min})$ of the Brandt $\omega$-extension of the semilattice $\boldsymbol{F}_{\min}$ and describe all shift-continuous feebly compact $T_1$-topologies on the semigroup $\mathscr{B}_{\omega}^{\Rsh}(\boldsymbol{F}_{\min})$. In particulary we prove that every shift-continuous feebly compact $T_1$-topology $\tau$ on  $\mathscr{B}_{\omega}^{\Rsh}(\boldsymbol{F}_{\min})$ is compact and moreover in this case  the space $(\mathscr{B}_{\omega}^{\Rsh}(\boldsymbol{F}_{\min}),\tau)$ is homeomorphic to the one-point Alexandroff compactification of the discrete countable space $\mathfrak{D}(\omega)$. We study the closure of $\boldsymbol{B}_{\omega}^{\mathscr{F}}$ in a semitopological semigroup. In particularly we show that $\boldsymbol{B}_{\omega}^{\mathscr{F}}$ is algebraically complete in the class of Hausdorff semitopological inverse semigroups with continuous inversion, and a Hausdorff topological inverse semigroup $\boldsymbol{B}_{\omega}^{\mathscr{F}}$ is closed in any Hausdorff topological semigroup if and only if the band $E(\boldsymbol{B}_{\omega}^{\mathscr{F}})$ is compact.
\end{abstract}

\maketitle

%\tableofcontents

\section{\textbf{Introduction, motivation and main definitions}}

We shall follow the terminology of~\cite{Carruth-Hildebrant-Koch-1983, Clifford-Preston-1961, Clifford-Preston-1967, Engelking-1989, Ruppert-1984}. By $\omega$ we denote the set of all non-negative integers.

Let $\mathscr{P}(\omega)$ be  the family of all subsets of $\omega$. For any $F\in\mathscr{P}(\omega)$ and $n,m\in\omega$ we put
\begin{equation*}
n-m+F=\{n-m+k\colon k\in F\}
\end{equation*}
This definition implies that $n-m+F=\varnothing$ if $F=\varnothing$. A subfamily $\mathscr{F}\subseteq\mathscr{P}(\omega)$ is called \emph{${\omega}$-closed} if $F_1\cap(-n+F_2)\in\mathscr{F}$ for all $n\in\omega$ and $F_1,F_2\in\mathscr{F}$.

A semigroup $S$ is called {\it inverse} if for any
element $x\in S$ there exists a unique $x^{-1}\in S$ (called the {\it inverse of} $x$) such that
$xx^{-1}x=x$ and $x^{-1}xx^{-1}=x^{-1}$.  If $S$ is an inverse
semigroup, then the function $\operatorname{inv}\colon S\to S$
which assigns to every element $x$ of $S$ its inverse element
$x^{-1}$ is called the {\it inversion}.

If $S$ is a semigroup, then we shall denote the subset of all
idempotents in $S$ by $E(S)$. If $S$ is an inverse semigroup, then
$E(S)$ is closed under multiplication and we shall refer to $E(S)$ as a
\emph{band} (or the \emph{band of} $S$). The semigroup
operation of $S$ determines the following partial order $\preccurlyeq$
on $E(S)$: $e\preccurlyeq f$ if and only if $ef=fe=e$. This order is
called the {\em natural partial order} on $E(S)$. A \emph{semilattice} is a commutative semigroup of idempotents. By $(\omega,\min)$ or $\omega_{\min}$ we denote the set $\omega$ with the semilattice operation $x\cdot y=\min\{x,y\}$.

If $S$ is an inverse semigroup then the semigroup operation on $S$ determines the following partial order $\preccurlyeq$
on $S$: $s\preccurlyeq t$ if and only if there exists $e\in E(S)$ such that $s=te$. This order is
called the {\em natural partial order} on $S$ \cite{Wagner-1952}.

The bicyclic monoid ${\mathscr{C}}(p,q)$ is the semigroup with the identity $1$ generated by two elements $p$ and $q$ subjected only to the condition $pq=1$. The semigroup operation on ${\mathscr{C}}(p,q)$ is determined as
follows:
\begin{equation*}
    q^kp^l\cdot q^mp^n=q^{k+m-\min\{l,m\}}p^{l+n-\min\{l,m\}}.
\end{equation*}
It is well known that the bicyclic monoid ${\mathscr{C}}(p,q)$ is a bisimple (and hence simple) combinatorial $E$-unitary inverse semigroup and every non-trivial congruence on ${\mathscr{C}}(p,q)$ is a group congruence \cite{Clifford-Preston-1961}.

On the set $\boldsymbol{B}_{\omega}=\omega\times\omega$ we define a semigroup operation ``$\cdot$'' in the following way
\begin{equation*}
  (i_1,j_1)\cdot(i_2,j_2)=
  \left\{
    \begin{array}{ll}
      (i_1-j_1+i_2,j_2), & \hbox{if~} j_1<i_2;\\
      (i_1,j_2),         & \hbox{if~} j_1=i_2;\\
      (i_1,j_1-i_2+j_2), & \hbox{if~} j_1>i_2.
    \end{array}
  \right.
\end{equation*}
It is well known that the semigroup $\boldsymbol{B}_{\omega}$ is isomorphic to the bicyclic monoid by the mapping $\mathfrak{h}\colon \mathscr{C}(p,q)\to \boldsymbol{B}_{\omega}$, $q^kp^l\mapsto (k,l)$ (see: \cite[Section~1.12]{Clifford-Preston-1961} or \cite[Exercise IV.1.11$(ii)$]{Petrich-1984}).

A {\it topological} (\emph{semitopological}) {\it semigroup} is a topological space together with a continuous (separately continuous) semigroup operation. If $S$ is a~semigroup and $\tau$ is a topology on $S$ such that $(S,\tau)$ is a topological semigroup, then we shall call $\tau$ a \emph{semigroup} \emph{topology} on $S$, and if $\tau$ is a topology on $S$ such that $(S,\tau)$ is a semitopological semigroup, then we shall call $\tau$ a \emph{shift-continuous} \emph{topology} on~$S$. An inverse topological semigroup with the continuous inversion is called a \emph{topological inverse semigroup}. If $S$ is an~inverse semigroup and $\tau$ is a topology on $S$ such that $(S,\tau)$ is a topological inverse semigroup, then we shall call $\tau$ a \emph{semigroup inverse topology} on $S$.

Next we shall describe the construction which is introduced in \cite{Gutik-Mykhalenych=2020}.

Let $\boldsymbol{B}_{\omega}$ be the bicyclic monoid and $\mathscr{F}$ be an ${\omega}$-closed subfamily of $\mathscr{P}(\omega)$. On the set $\boldsymbol{B}_{\omega}\times\mathscr{F}$ we define the semigroup operation ``$\cdot$'' in the following way
\begin{equation*}
  (i_1,j_1,F_1)\cdot(i_2,j_2,F_2)=
  \left\{
    \begin{array}{ll}
      (i_1-j_1+i_2,j_2,(j_1-i_2+F_1)\cap F_2), & \hbox{if~} j_1<i_2;\\
      (i_1,j_2,F_1\cap F_2),                   & \hbox{if~} j_1=i_2;\\
      (i_1,j_1-i_2+j_2,F_1\cap (i_2-j_1+F_2)), & \hbox{if~} j_1>i_2.
    \end{array}
  \right.
\end{equation*}
By \cite{Gutik-Mykhalenych=2020}, if the family $\mathscr{F}\subseteq\mathscr{P}(\omega)$ is ${\omega}$-closed, then $(\boldsymbol{B}_{\omega}\times\mathscr{F},\cdot)$ is a semigroup. Moreover, if an ${\omega}$-closed family  $\mathscr{F}\subseteq\mathscr{P}(\omega)$ contains the empty set $\varnothing$, then the set
\begin{equation*}
  \boldsymbol{I}=\{(i,j,\varnothing)\colon i,j\in\omega\}
\end{equation*}
is an ideal of the semigroup $(\boldsymbol{B}_{\omega}\times\mathscr{F},\cdot)$. For any ${\omega}$-closed family $\mathscr{F}\subseteq\mathscr{P}(\omega)$ the following semigroup
\begin{equation*}
  \boldsymbol{B}_{\omega}^{\mathscr{F}}=
\left\{
  \begin{array}{ll}
    (\boldsymbol{B}_{\omega}\times\mathscr{F},\cdot)/\boldsymbol{I}, & \hbox{if~} \varnothing\in\mathscr{F};\\
    (\boldsymbol{B}_{\omega}\times\mathscr{F},\cdot), & \hbox{if~} \varnothing\notin\mathscr{F}
  \end{array}
\right.
\end{equation*}
is defined in \cite{Gutik-Mykhalenych=2020}. The semigroup $\boldsymbol{B}_{\omega}^{\mathscr{F}}$ generalizes the bicyclic monoid and the countable semigroup of matrix units. It is proven in \cite{Gutik-Mykhalenych=2020} that $\boldsymbol{B}_{\omega}^{\mathscr{F}}$ is combinatorial inverse semigroup and Green's relations, the natural partial order on $\boldsymbol{B}_{\omega}^{\mathscr{F}}$ and its set of idempotents are described. The criteria of simplicity, $0$-simplicity, bisimplicity, $0$-bisimplicity of the semigroup $\boldsymbol{B}_{\omega}^{\mathscr{F}}$ and when $\boldsymbol{B}_{\omega}^{\mathscr{F}}$ has the identity, is isomorphic to the bicyclic semigroup or the countable semigroup of matrix units are given. In particular in \cite{Gutik-Mykhalenych=2020} it is proved that the semigroup $\boldsymbol{B}_{\omega}^{\mathscr{F}}$ is isomorphic to the semigrpoup of ${\omega}{\times}{\omega}$-matrix units if and only if $\mathscr{F}$ consists of sets of cardinality $\leqslant 1$ in $\omega$.

Let $\mathscr{F}$ be some family of cardinality $\leqslant 1$ in $\omega$. In this case we shall say that $\mathscr{F}$ is the \emph{family of atomic subsets} of $\omega$. It is obvious that if $\mathscr{F}=\{\varnothing\}$ then the semigroup $\boldsymbol{B}_{\omega}^{\mathscr{F}}$ is trivial and hence in this paper we assume that the family $\mathscr{F}$ contains at least one  singleton subset of $\omega$. It is obvious that in this case $\mathscr{F}$ is an ${\omega}$-closed subfamily of $\mathscr{P}(\omega)$ and hence $\boldsymbol{B}_{\omega}^{\mathscr{F}}$ is an inverse semigroup with zero. Later by $\boldsymbol{0}$ we denote the zero of $\boldsymbol{B}_{\omega}^{\mathscr{F}}$ and by $(i,j,\{k\})$ a non-zero element of $\boldsymbol{B}_{\omega}^{\mathscr{F}}$ for some $i,j\in\omega$, $\{k\}\in \mathscr{F}$.

We put $\boldsymbol{F}=\displaystyle\bigcup\mathscr{F}$. Since the semilattice $(\omega,\min)$ is linearly ordered, the set $\boldsymbol{F}$ with the binary operation $xy=\min\{x,y\}$ is a subsemilattice of $(\omega,\min)$ and later by $\boldsymbol{F}_{\min}$ we shall denote the set $\boldsymbol{F}$ with the  semilattice operation inherited  from $(\omega,\min)$.

We need the following construction from \cite{Gutik=1999}.

Let $S$ be a semigroup with zero and $\lambda\geqslant 1$ be a cardinal. On the set
$B_{\lambda}(S)=\left(\lambda\times S\times\lambda\right)\sqcup\{ \mathscr{O}\}$ we define a  semigroup operation as follows
\begin{equation*}
 (\alpha,s,\beta)\cdot(\gamma, t, \delta)=
\left\{
  \begin{array}{cl}
    (\alpha, st, \delta), & \hbox{if~} \beta=\gamma; \\
    \mathscr{O},          & \hbox{if~} \beta\ne \gamma
  \end{array}
\right.
\end{equation*}
and
\begin{equation*}
(\alpha, s, \beta)\cdot \mathscr{O}=\mathscr{O}\cdot(\alpha, s, \beta)=\mathscr{O}\cdot \mathscr{O}=\mathscr{O},
\end{equation*}
for all $\alpha, \beta, \gamma, \delta\in \lambda$ and $s, t\in S$. The semigroup $\mathscr{B}_\lambda(S)$ is called the {\it Brandt $\lambda$-extension of the semigroup} $S$~\cite{Gutik=1999}. Algebraic properties of $\mathscr{B}_\lambda(S)$ and its generalization the Brandt $\lambda^0$-extension  $\mathscr{B}^0_\lambda(S)$ are studied in \cite{Gutik=1999, Gutik=2016, Gutik-Pavlyk=2006, Gutik-Repovs=2010}.

In this paper we study the semigroup $\boldsymbol{B}_{\omega}^{\mathscr{F}}$ for a family $\mathscr{F}$ of atomic subsets of $\omega$. We show that $\boldsymbol{B}_{\omega}^{\mathscr{F}}$ is isomorphic to the  subsemigroup $\mathscr{B}_{\omega}^{\Rsh}(\boldsymbol{F}_{\min})$ of the Brandt $\omega$-extension of the semilattice $\boldsymbol{F}_{\min}$ and describe all shift-continuous feebly compact $T_1$-topologies on the semigroup $\mathscr{B}_{\omega}^{\Rsh}(\boldsymbol{F}_{\min})$. In particular, we prove that every shift-continuous feebly compact $T_1$-topology $\tau$ on  $\mathscr{B}_{\omega}^{\Rsh}(\boldsymbol{F}_{\min})$ is compact and moreover in this case  the space $(\mathscr{B}_{\omega}^{\Rsh}(\boldsymbol{F}_{\min}),\tau)$ is homeomorphic to the one-point Alexandroff compactification of the discrete countable space $\mathfrak{D}(\omega)$. We study the closure of $\boldsymbol{B}_{\omega}^{\mathscr{F}}$ in a semitopological semigroup. In particularly we show that $\boldsymbol{B}_{\omega}^{\mathscr{F}}$ is algebraically complete in the class of Hausdorff semitopological inverse semigroups with continuous inversion, and a Hausdorff topological inverse semigroup $\boldsymbol{B}_{\omega}^{\mathscr{F}}$ is closed in any Hausdorff topological semigroup if and only if the band $E(\boldsymbol{B}_{\omega}^{\mathscr{F}})$ is compact.

Later in this paper we  assume that $\mathscr{F}$ is a non-trivial family of atomic subsets of $\omega$, i.e., $\mathscr{F}$ contains at least one nontrivial singleton subset of $\omega$.

\section{\textbf{Algebraic properties of the semigroup $\boldsymbol{B}_{\omega}^{\mathscr{F}}$}}

Proposition 2 of \cite{Gutik-Mykhalenych=2020} implies the following proposition  which describing the natural partial order on $\boldsymbol{B}_{\omega}^{\mathscr{F}}$.

\begin{proposition}\label{proposition-2.1}
Let $(i_1,j_1,\{k_1\})$ and $(i_2,j_2,\{k_2\})$ be non-zero elements of the semigroup $\boldsymbol{B}_{\omega}^{\mathscr{F}}$. Then $(i_1,j_1,\{k_1\})\preccurlyeq(i_2,j_2,\{k_2\})$ if and only if
\begin{equation*}
k_2-k_1=i_1-i_2=j_1-j_2=p
\end{equation*}
for some $p\in\omega$.
\end{proposition}

Since the set $\omega$ is well ordered by the usual order  we enumerate the set $\boldsymbol{F}=\{k_i\colon i\in \omega\}$ in the following way $k_0<k_1<\cdots<k_n<k_{n+1}<\cdots$. It is obvious that the set $\boldsymbol{F}$ is finite if and only if $\boldsymbol{F}$ contains the maximum.

Proposition \ref{proposition-2.1} implies the structure of maximal chains in $\boldsymbol{B}_{\omega}^{\mathscr{F}}$ with the respect to its natural partial order

\begin{corollary}\label{corollary-2.2}
Let $i,j$ be arbitrary elements of $\omega$. Then in the case when the set $\boldsymbol{F}$ is infinite then the following finite series
\begin{align*}
  \boldsymbol{0}&\preccurlyeq(i,j,\{k_0\}); \\
  \boldsymbol{0}&\preccurlyeq(i+k_1-k_0,j+k_1-k_0,\{k_0\})\preccurlyeq(i,j,\{k_1\});\\
  \boldsymbol{0}&\preccurlyeq(i+k_1-k_0,j+k_1-k_0,\{k_0\})\preccurlyeq(i+k_2-k_1,j+k_2-k_1,\{k_1\})\preccurlyeq(i,j,\{k_2\});\\
                & \qquad \cdots\qquad  \cdots\qquad \cdots\qquad \cdots \qquad \cdots \qquad \cdots\\
  \boldsymbol{0}&\preccurlyeq(i+k_1-k_0,j+k_1-k_0,\{k_0\})\preccurlyeq(i+k_2-k_1,j+k_2-k_1,\{k_1\})\preccurlyeq\cdots\preccurlyeq\\
                &\quad\preccurlyeq(i+k_{n+1}-k_{n},j+k_{n+1}-k_{n},\{k_{n}\})\preccurlyeq(i,j,\{k_{n+1}\});\\
                & \qquad \cdots\qquad  \cdots\qquad \cdots\qquad \cdots \qquad \cdots \qquad \cdots\qquad \cdots \qquad \cdots
\end{align*}
describes maximal chains in the semigroup $\boldsymbol{B}_{\omega}^{\mathscr{F}}$ and in the case when the set $\boldsymbol{F}$ is finite and contains maximum $k_n$ then the following finite series
\begin{align*}
  \boldsymbol{0}&\preccurlyeq(i,j,\{k_0\}); \\
  \boldsymbol{0}&\preccurlyeq(i+k_1-k_0,j+k_1-k_0,\{k_0\})\preccurlyeq(i,j,\{k_1\});\\
  \boldsymbol{0}&\preccurlyeq(i+k_1-k_0,j+k_1-k_0,\{k_0\})\preccurlyeq(i+k_2-k_1,j+k_2-k_1,\{k_1\})\preccurlyeq(i,j,\{k_2\});\\
                & \qquad \cdots\qquad  \cdots\qquad \cdots\qquad \cdots \qquad \cdots \qquad \cdots\\
  \boldsymbol{0}&\preccurlyeq(i+k_1-k_0,j+k_1-k_0,\{k_0\})\preccurlyeq(i+k_2-k_1,j+k_2-k_1,\{k_1\})\preccurlyeq\cdots\preccurlyeq\\
                &\quad\preccurlyeq(i+k_{n}-k_{n-1},j+k_{n}-k_{n-1},\{k_{n}\})\preccurlyeq(i,j,\{k_{n}\})
\end{align*}
describes maximal chains in the semigroup $\boldsymbol{B}_{\omega}^{\mathscr{F}}$.
\end{corollary}

We define a map $\mathfrak{f}\colon\boldsymbol{B}_{\omega}^{\mathscr{F}}\to\mathscr{B}_{\omega}(\boldsymbol{F}_{\min})$ by the formulae
\begin{equation}\label{eq-2.1}
  \mathfrak{f}(i,j,\{k\})=(i+k,k,j+k) \qquad \hbox{and} \qquad (\boldsymbol{0})\mathfrak{f}=\mathscr{O},
\end{equation}
for $i,j\in\omega$ and $\{k\}\in\mathscr{F}\setminus \{\varnothing\}$.

\begin{proposition}\label{proposition-2.3}
The map $\mathfrak{f}\colon\boldsymbol{B}_{\omega}^{\mathscr{F}}\to\mathscr{B}_{\omega}(\boldsymbol{F}_{\min})$ is an isomorphic embedding.
\end{proposition}

\begin{proof}[\textsl{Proof}]
It is obvious that the map  $\mathfrak{f}$ which is defined by formulae \eqref{eq-2.1} is injective.

For  arbitrary $(i_1,j_1,\{k_1\}),(i_2,j_2,\{k_2\})\in \boldsymbol{B}_{\omega}^{\mathscr{F}}$  we have that
\begin{align*}
  \mathfrak{f}((i_1,j_1,&\{k_1\})\cdot(i_2,j_2,\{k_2\}))=\\
&=\left\{
  \begin{array}{cl}
    \mathfrak{f}(i_1-j_1+i_2,j_2,\{k_2\}), & \hbox{if~} j_1<i_2 \hbox{~and~} j_1+k_1=i_2+k_2;\\
    \mathfrak{f}(i_1,j_2,\{k_1\}),         & \hbox{if~} j_1=i_2 \hbox{~and~} k_1=k_2;\\
    \mathfrak{f}(i_1,j_1-i_2+j_2,\{k_1\}), & \hbox{if~} j_1>i_2 \hbox{~and~} j_1+k_1=i_2+k_2;\\
    \mathfrak{f}(\boldsymbol{0}),          & \hbox{if~} j_1+k_1\neq i_2+k_2
  \end{array}
\right.=\\
    & =\left\{
  \begin{array}{cl}
    (i_1-j_1+i_2+k_2,k_2,j_2+k_2), & \hbox{if~} j_1<i_2 \hbox{~and~} j_1+k_1=i_2+k_2;\\
    (i_1+k_1,k_1,j_2+k_1),         & \hbox{if~} j_1=i_2 \hbox{~and~} k_1=k_2;\\
    (i_1+k_1,k_1,j_1-i_2+j_2+k_1), & \hbox{if~} j_1>i_2 \hbox{~and~} j_1+k_1=i_2+k_2;\\
    \mathscr{O},                   & \hbox{if~} j_1+k_1\neq i_2+k_2
  \end{array}
\right.=\\
    & =\left\{
  \begin{array}{cl}
    (i_1+k_1,k_2,j_2+k_2), & \hbox{if~} j_1<i_2 \hbox{~and~} j_1+k_1=i_2+k_2;\\
    (i_1+k_1,k_1,j_2+k_2), & \hbox{if~} j_1=i_2 \hbox{~and~} k_1=k_2;\\
    (i_1+k_1,k_1,j_2+k_2), & \hbox{if~} j_1>i_2 \hbox{~and~} j_1+k_1=i_2+k_2;\\
    \mathscr{O},           & \hbox{if~} j_1+k_1\neq i_2+k_2,
  \end{array}
\right.
\end{align*}
and
\begin{align*}
  \mathfrak{f}((i_1,j_1,\{k_1\})&\cdot\mathfrak{f}(i_2,j_2,\{k_2\}))=(i_1+k_1,k_1,j_1+k_1)\cdot(i_2+k_2,k_2,j_2+k_2)=\\
&=
\left\{
  \begin{array}{cl}
    (i_1+k_1,\min\{k_1,k_2\},j_2+k_2), & \hbox{if~} j_1+k_1=i_2+k_2; \\
    \mathscr{O},                       & \hbox{if~} j_1+k_1\neq i_2+k_2
  \end{array}
\right.=\\
    & =\left\{
  \begin{array}{cl}
    (i_1+k_1,k_2,j_2+k_2), & \hbox{if~} k_2<k_1 \hbox{~and~} j_1+k_1=i_2+k_2;\\
    (i_1+k_1,k_1,j_2+k_2), & \hbox{if~} k_2=k_1 \hbox{~and~} k_1=k_2;\\
    (i_1+k_1,k_1,j_2+k_2), & \hbox{if~} k_2>k_1 \hbox{~and~} j_1+k_1=i_2+k_2;\\
    \mathscr{O},           & \hbox{if~} j_1+k_1\neq i_2+k_2,
  \end{array}
\right.=\\
    & =\left\{
  \begin{array}{cl}
    (i_1+k_1,k_2,j_2+k_2), & \hbox{if~} j_1<i_2 \hbox{~and~} j_1+k_1=i_2+k_2;\\
    (i_1+k_1,k_1,j_2+k_2), & \hbox{if~} j_1=i_2 \hbox{~and~} k_1=k_2;\\
    (i_1+k_1,k_1,j_2+k_2), & \hbox{if~} j_1>i_2 \hbox{~and~} j_1+k_1=i_2+k_2;\\
    \mathscr{O},           & \hbox{if~} j_1+k_1\neq i_2+k_2.
  \end{array}
\right.
\end{align*}
Since $\boldsymbol{0}$ and $\mathscr{O}$ are the zeros of the semigroups $\boldsymbol{B}_{\omega}^{\mathscr{F}}$ and $\mathscr{B}_{\omega}(\boldsymbol{F}_{\min})$, respectively, the above equalities imply that the map $\mathfrak{f}\colon\boldsymbol{B}_{\omega}^{\mathscr{F}}\to\mathscr{B}_{\omega}(\boldsymbol{F}_{\min})$ is a homomorphism. This completes the proof of the proposition.
\end{proof}

Next we define
\begin{equation*}
  \mathscr{B}_{\omega}^{\Rsh}(\boldsymbol{F}_{\min})=\left\{\mathscr{O}\right\}\cup \left\{(i+k,k,j+k)\in\mathscr{B}_{\omega}(\boldsymbol{F}_{\min})\setminus\left\{\mathscr{O}\right\}\colon (i,j,\{k\})\in \boldsymbol{B}_{\omega}^{\mathscr{F}}\right\}.
\end{equation*}
%Simple verifications show that $\mathscr{B}_{\omega}^{\Rsh}(\boldsymbol{F}_{\min})$ is an inverse subsemigroup of $\mathscr{B}_{\omega}(\boldsymbol{F}_{\min})$.

Proposition \ref{proposition-2.3} implies

\begin{theorem}\label{theorem-2.4}
Let $\mathscr{F}^*$ be any family of atomic subsets of $\omega$.
Then the semigroup  $\boldsymbol{B}_{\omega}^{\mathscr{F}}$  is isomorphic to $\mathscr{B}_{\omega}^{\Rsh}(\boldsymbol{F}_{\min})$ by the mapping $\mathfrak{f}$.
\end{theorem}

\begin{proposition}\label{proposition-2.5}
Let $\mathscr{F}^*$ be any family of subsets of $\omega$ which contains a non-empty set, and $k_0=\min \displaystyle\bigcup\mathscr{F}^*$. Then the semigroup  $\boldsymbol{B}_{\omega}^{\mathscr{F}^*}$  is isomorphic to the semigroup $\boldsymbol{B}_{\omega}^{\mathscr{F}_0^*}$ where
\begin{equation*}
  \mathscr{F}_0^*=\left\{-k_0+F\colon F\in \mathscr{F}^*\right\}.
\end{equation*}
\end{proposition}

\begin{proof}[\textsl{Proof}]
Since the set $\omega$ with the usual order $\leqslant$ is well ordered, the number $k_0$ is well defined. This implies that the semigroup $\boldsymbol{B}_{\omega}^{\mathscr{F}_0^*}$ is well defined, because $F\subseteq\{n\in\omega\colon n\geqslant k_0 \}$ for any $F\in\mathscr{F}^*$. Without loss of generality we may assume that $\varnothing\in \mathscr{F}^*$, which implies that the semigroup $\boldsymbol{B}_{\omega}^{\mathscr{F}^*}$ has zero $\boldsymbol{0}$, and hence the semigroup $\boldsymbol{B}_{\omega}^{\mathscr{F}_0^*}$ has zero $\boldsymbol{0}$, too.

We define the map $\mathfrak{h}\colon \boldsymbol{B}_{\omega}^{\mathscr{F}^*}\to \boldsymbol{B}_{\omega}^{\mathscr{F}_0^*}$ in the following way
\begin{equation}\label{eq-2.2}
  \mathfrak{h}(i,j,\{k\})=(i-k_0,j-k_0,\{k-k_0\}) \qquad \hbox{and} \qquad (\boldsymbol{0})\mathfrak{h}=\boldsymbol{0}
\end{equation}
for $i,j\in\omega$ and $\{k\}\in\mathscr{F}^*\setminus \{\varnothing\}$. It is obvious that such defined map  $\mathfrak{h}$ is bijective.

For arbitrary $(i_1,j_1,\{k_1\}),(i_2,j_2,\{k_2\})\in \boldsymbol{B}_{\omega}^{\mathscr{F}^*}$  we have that
\begin{align*}
  \mathfrak{h}((i_1,&j_1,\{k_1\})\cdot(i_2,j_2,\{k_2\}))=\\
&=\left\{
  \begin{array}{cl}
    \mathfrak{h}(i_1-j_1+i_2,j_2,\{k_2\}), & \hbox{if~} j_1<i_2 \hbox{~and~} j_1+k_1=i_2+k_2;\\
    \mathfrak{h}(i_1,j_2,\{k_1\}),         & \hbox{if~} j_1=i_2 \hbox{~and~} k_1=k_2;\\
    \mathfrak{h}(i_1,j_1-i_2+j_2,\{k_1\}), & \hbox{if~} j_1>i_2 \hbox{~and~} j_1+k_1=i_2+k_2;\\
    \mathfrak{h}(\boldsymbol{0}),          & \hbox{if~} j_1+k_1\neq i_2+k_2
  \end{array}
\right.=\\
    & =\left\{
  \begin{array}{cl}
    (i_1-j_1+i_2-k_0,j_2-k_0,\{k_2-k_0\}), & \hbox{if~} j_1<i_2 \hbox{~and~} j_1+k_1=i_2+k_2;\\
    (i_1-k_0,j_2-k_0,\{k_1-k_0\}),         & \hbox{if~} j_1=i_2 \hbox{~and~} k_1=k_2;\\
    (i_1-k_0,j_1-i_2+j_2-k_0,\{k_1-k_0\}), & \hbox{if~} j_1>i_2 \hbox{~and~} j_1+k_1=i_2+k_2;\\
    \boldsymbol{0},                        & \hbox{if~} j_1+k_1\neq i_2+k_2
  \end{array}
\right.
\end{align*}
and
\begin{align*}
  &\mathfrak{h}(i_1,j_1,\{k_1\})\cdot\mathfrak{h}(i_2,j_2,\{k_2\})=\\
  &=
(i_1-k_0,j_1-k_0,\{k_1-k_0\})\cdot(i_2-k_0,j_2-k_0,\{k_2-k_0\})=\\
&=\left\{
  \begin{array}{cl}
    (i_1{-}k_0{-}(j_1{-}k_0){+}i_2{-}k_0,j_2{-}k_0,\{k_2{-}k_0\}), & \hbox{if~} j_1-k_0<i_2-k_0 \hbox{~and~}\\
                                                     & \quad j_1{-}k_0{+}k_1{-}k_0{=}i_2{-}k_0{+}k_2{-}k_0;\\
    (i_1-k_0,j_2-k_0,\{k_1-k_0\}),                   & \hbox{if~} j_1-k_0=i_2-k_0 \hbox{~and~}\\
                                                     & \qquad k_1-k_0=k_2-k_0;\\
    (i_1{-}k_0,j_1{-}k_0{-}(i_2{-}k_0){+}j_2{-}k_0,\{k_1{-}k_0\}), & \hbox{if~} j_1-k_0>i_2-k_0 \hbox{~and~}\\
                                                     & \quad  j_1{-}k_0{+}k_1{-}k_0{=}i_2{-}k_0{+}k_2{-}k_0;\\
    \boldsymbol{0},                                  & \hbox{if~} j_1{-}k_0{+}k_1{-}k_0{\neq} i_2{-}k_0{+}k_2{-}k_0
  \end{array}
\right.\\
&=\left\{
  \begin{array}{cl}
    (i_1-j_1+i_2-k_0,j_2-k_0,\{k_2-k_0\}), & \hbox{if~} j_1<i_2 \hbox{~and~} j_1+k_1=i_2+k_2;\\
    (i_1-k_0,j_2-k_0,\{k_1-k_0\}),         & \hbox{if~} j_1=i_2 \hbox{~and~} k_1=k_2;\\
    (i_1-k_0,j_1-i_2+j_2-k_0,\{k_1-k_0\}), & \hbox{if~} j_1>i_2 \hbox{~and~} j_1+k_1=i_2+k_2;\\
    \boldsymbol{0},                                  & \hbox{if~} j_1+k_1\neq i_2+k_2.
  \end{array}
\right.
\end{align*}
Since $\boldsymbol{0}$ is the zero of both semigroups $\boldsymbol{B}_{\omega}^{\mathscr{F}^*}$ and $\boldsymbol{B}_{\omega}^{\mathscr{F}_0^*}$, the above equalities imply that such defined  map $\mathfrak{h}\colon \boldsymbol{B}_{\omega}^{\mathscr{F}^*}\to \boldsymbol{B}_{\omega}^{\mathscr{F}_0^*}$ is a homomorphism. %This completes the proof of the proposition.
\end{proof}

\begin{theorem}\label{theorem-2.6}
Let $\mathscr{F}^1$ and $\mathscr{F}^2$ be some families of atomic subsets of $\omega$. Then the semigroups $\boldsymbol{B}_{\omega}^{\mathscr{F}^1}$ and $\boldsymbol{B}_{\omega}^{\mathscr{F}^2}$ are isomorphic if and only if there exists an integer $n$ such that
\begin{equation*}
  \mathscr{F}^1=\left\{n+F\colon F\in\mathscr{F}^2\right\}.
\end{equation*}
\end{theorem}

\begin{proof}[\textsl{Proof}]
The implication  $(\Leftarrow)$ follows from Proposition~\ref{proposition-2.5}.

$(\Rightarrow)$ Put $\boldsymbol{F}^1=\bigcup\mathscr{F}^1$ and $\boldsymbol{F}^2=\bigcup\mathscr{F}^2$. By Proposition~\ref{proposition-2.5}. without loss of generality we may assume that $0\in \boldsymbol{F}^1\cap\boldsymbol{F}^2$, i.e., $\{0\}\in\mathscr{F}^1$ and $\{0\}\in\mathscr{F}^2$.

Suppose to the contrary that the semigroups $\boldsymbol{B}_{\omega}^{\mathscr{F}^1}$ and $\boldsymbol{B}_{\omega}^{\mathscr{F}^2}$ are isomorphic but $\mathscr{F}^1\neq\mathscr{F}^2$. Since $\mathscr{F}^1$ and $\mathscr{F}^2$ are some families of atomic subsets of $\omega$, we get that $\boldsymbol{F}^1\neq\boldsymbol{F}^2$. Hence without loss of generality we may assume that there exists the minimum positive integer $m$ of the set $\boldsymbol{F}^1$ such that $m\notin \boldsymbol{F}^2$. Put
\begin{equation*}
\widetilde{\boldsymbol{F}}=\left\{k\in\boldsymbol{F}^2\colon k<m\right\}.
\end{equation*}
We enumerate the set $\widetilde{\boldsymbol{F}}=\left\{k_0,k_1,\ldots,k_n\right\}$ in the following way
\begin{equation*}
k_0=0<k_1<\cdots<k_n.
\end{equation*}
Then we have that $\widetilde{\boldsymbol{F}}\subset \boldsymbol{F}^1$.

By Lemma~2 of \cite{Gutik-Mykhalenych=2020} a non-zero element $(i,j,\{k\})$ of the semigroup $\boldsymbol{B}_{\omega}^{\mathscr{F}^1}$ (or $\boldsymbol{B}_{\omega}^{\mathscr{F}^2}$) is an idempotent if and only if $i=j$. This and Corollary~\ref{corollary-2.2} imply the semigroup $\boldsymbol{B}_{\omega}^{\mathscr{F}^1}$ contains exactly $m-k_n$ distinct chains (or a chain) of idempotents of the length $k_n+2$, but the semigroup $\boldsymbol{B}_{\omega}^{\mathscr{F}^1}$ contains at least $m-k_n+1$ distinct chains  of idempotents of the length $k_n+2$. This contradicts that the semigroups $\boldsymbol{B}_{\omega}^{\mathscr{F}^1}$ and $\boldsymbol{B}_{\omega}^{\mathscr{F}^2}$ are isomorphic. The obtained contradiction implies the implication.
\end{proof}

For any $i,j\in \omega$ we denote
\begin{equation*}
  \boldsymbol{F}_{\min}^{(i,j)_{\Rsh}}=\left\{(i,k,j)\colon(i,k,j)\in\mathscr{B}_{\omega}^{\Rsh}(\boldsymbol{F}_{\min})\right\}
\end{equation*}
and
\begin{equation*}
  \omega_{\min}^{(i,j)}=\left\{(i,k,j)\colon(i,k,j)\in\mathscr{B}_{\omega}(\omega_{\min})\right\},
\end{equation*}
where by $\omega_{\min}$ we denote the semilattice $(\omega,\min)$.

\begin{lemma}\label{lemma-2.7}
In the semigroup $\boldsymbol{B}_{\omega}^{\mathscr{F}}$ both equations $A\cdot X=B$ and $X\cdot A=B$ have only finitely many solutions for $B\neq\boldsymbol{0}$.
\end{lemma}

\begin{proof}[\textsl{Proof}]
We show that the equation $A\cdot X=B$ has finitely many solutions for $B\neq\mathscr{O}$ in the semigroup $\mathscr{B}_{\omega}^{\Rsh}(\boldsymbol{F}_{\min})$. In the case of the equation $X\cdot A=B$ the proof is similar.

We denote
\begin{equation*}
  A=(i_A,k_A,j_A), \qquad X=(i_X,k_X,j_X) \qquad \hbox{and} \qquad B=(i_B,k_B,j_B),
\end{equation*}
where $(i_X,k_X,j_X)$ is a variable, $(i_A,k_A,j_A)$ and $(i_B,k_B,j_B)$ are constants of the equation
\begin{equation}\label{eq-2.3}
  (i_A,k_A,j_A)\cdot(i_X,k_X,j_X)=(i_B,k_B,j_B).
\end{equation}
First we establish the solution of equation \eqref{eq-2.3} in the Brandt $\omega$-extension $\mathscr{B}_{\omega}(\omega_{\min})$ of the semilattice $\omega_{\min}$. The semigroup operation in $\mathscr{B}_{\omega}(\omega_{\min})$ implies that equation \eqref{eq-2.3} has a non-empty set of solutions if and only if $k_B\preccurlyeq k_A$ in $\omega_{\min}$ and $i_A=i_B$. Hence we have that the set of solutions of \eqref{eq-2.3} is a subset of $\omega_{\min}^{(j_A,j_B)}$. This implies that the set of solutions of equation \eqref{eq-2.3}  is a subset of $\boldsymbol{F}_{\min}^{(j_A,j_B)_{\Rsh}}$. This and Theorem~\ref{theorem-2.4} imply the statement of the lemma.
\end{proof}

\section{\textbf{On topogizations of the semigroup $\mathscr{B}_{\omega}^{\Rsh}(\boldsymbol{F}_{\min})$}}

By Proposition~\ref{proposition-2.5} for any family $\mathscr{F}$ of atomic subsets of $\omega$ the semigroup  $\boldsymbol{B}_{\omega}^{\mathscr{F}}$  is isomorphic to the semigroup $\boldsymbol{B}_{\omega}^{\mathscr{F}_0}$ where $\mathscr{F}_0$ is a family of atomic subsets of $\omega$ such that $0\in\displaystyle\bigcup\mathscr{F}_0$. Hence later we shall assume that $0\in \boldsymbol{F}$, i.e., $(i,0,i)\in \mathscr{B}_{\omega}^{\Rsh}(\boldsymbol{F}_{\min})$ for any $i,j\in\omega$.

\begin{proposition}\label{proposition-3.1}
Let $\tau$ be a shift-continuous $T_1$-topology on the semigroup $\mathscr{B}_{\omega}^{\Rsh}(\boldsymbol{F}_{\min})$. Then every non-zero element of $\mathscr{B}_{\omega}^{\Rsh}(\boldsymbol{F}_{\min})$ is an isolated point in $\left(\mathscr{B}_{\omega}^{\Rsh}(\boldsymbol{F}_{\min}),\tau\right)$.
\end{proposition}

\begin{proof}[\textsl{Proof}]
Fix arbitrary $i,j\in \omega$. Since
\begin{equation*}
(i,0,i)\cdot(i,0,j)\cdot(j,0,j)=(i,0,j)
\end{equation*}
the assumption of the proposition implies that for any open neighbourhood $W_{(i,0,j)}\not\ni\mathscr{O}$ of the point $(i,0,j)$ there exists its open neighbourhood $V_{(i,0,j)}$ in the topological space $\left(\mathscr{B}_{\omega}^{\Rsh}(\boldsymbol{F}_{\min}),\tau\right)$ such that
\begin{equation*}
(i,0,i)\cdot V_{(i,0,j)}\cdot(j,0,j)\subseteq W_{(i,0,j)}.
\end{equation*}
The definition of the semigroup operation on $\mathscr{B}_{\omega}^{\Rsh}(\boldsymbol{F}_{\min})$ implies that $V_{(i,0,j)}\subseteq \boldsymbol{F}_{\min}^{(i,j)_{\Rsh}}$. Then  $\boldsymbol{F}_{\min}^{(i,j)_{\Rsh}}$ is an open subset of the set $\left(\mathscr{B}_{\omega}^{\Rsh}(\boldsymbol{F}_{\min}),\tau\right)$ because it is the full preimage of $V_{(i,0,j)}$ under the mapping
\begin{equation*}
  \mathfrak{h}\colon \mathscr{B}_{\omega}^{\Rsh}(\boldsymbol{F}_{\min})\to\mathscr{B}_{\omega}^{\Rsh}(\boldsymbol{F}_{\min}),\; x\mapsto (i,0,i)\cdot x\cdot(j,0,j).
\end{equation*}
By Corollary~\ref{corollary-2.2} the set $\boldsymbol{F}_{\min}^{(i,j)_{\Rsh}}$ is finite, which implies the statement of the proposition.
\end{proof}

Next we shall show that the semigroup $\mathscr{B}_{\omega}^{\Rsh}(\boldsymbol{F}_{\min})$ admits a compact shift-continuous Hausdorff topology.

\begin{example}\label{example-3.2}
A topology $\tau_{\mathrm{Ac}}$ on the semigroup $\mathscr{B}_{\omega}^{\Rsh}(\boldsymbol{F}_{\min})$ is defined as follows:
\begin{enumerate}
\item[a)] all nonzero elements of $\mathscr{B}_{\omega}^{\Rsh}(\boldsymbol{F}_{\min})$ are isolated points in $\left(\mathscr{B}_{\omega}^{\Rsh}(\boldsymbol{F}_{\min}),\tau_{\mathrm{Ac}}\right)$;
\item[b)] the family
\begin{align*}
  {\mathscr B}_{\mathrm{Ac}}(\mathscr{O})=\Big\{ U_{(i_1,j_1),\ldots,(i_n,j_n)}& =\mathscr{B}_{\omega}^{\Rsh}(\boldsymbol{F}_{\min})\setminus \left(\boldsymbol{F}_{\min}^{(i_1,j_1)_{\Rsh}}\cup\cdots\cup\boldsymbol{F}_{\min}^{(i_n,j_n)_{\Rsh}}\right) \colon \\
    & n,i_1,j_1,\ldots,i_n,j_n\in\omega\Big\}
\end{align*}
is the base of the topology $\tau_{\mathrm{Ac}}$ at the point $\mathscr{O}\in \mathscr{B}_{\omega}^{\Rsh}(\boldsymbol{F}_{\min})$.
\end{enumerate}
Corollary~\ref{corollary-2.2} implies that the set $\boldsymbol{F}_{\min}^{(i,j)_{\Rsh}}$ is finite for any $i,j\in\omega$ which implies that the topological space $\left(\mathscr{B}_{\omega}^{\Rsh}(\boldsymbol{F}_{\min}),\tau_{\mathrm{Ac}}\right)$ is homeomorphic to the one-point Alexandroff compactification of the discrete space $\mathscr{B}_{\omega}^{\Rsh}(\boldsymbol{F}_{\min})\setminus\left\{\mathscr{O}\right\}$.
\end{example}

\begin{proposition}\label{proposition-3.3}
$\left(\mathscr{B}_{\omega}^{\Rsh}(\boldsymbol{F}_{\min}),\tau_{\mathrm{Ac}}\right)$ is a Hausdorff compact semitopological semigroup with continuous inversion.
\end{proposition}

\begin{proof}[\textsl{Proof}]
It is obvious that the topology $\tau_{\mathrm{Ac}}$ is Hausdorff and compact.

Fix any $U_{(i_1,j_1),\ldots,(i_n,j_n)}\in {\mathscr B}_{\mathrm{Ac}}(\mathscr{O})$ and $(i,k,j), (l,m,p)\in\mathscr{B}_{\omega}^{\Rsh}(\boldsymbol{F}_{\min})\setminus \left\{\mathscr{O}\right\}$. Put
\begin{equation*}
\boldsymbol{K}=\{i,i_1,\ldots,i_n,j,j_1,\ldots,j_n\} \qquad \hbox{and} \qquad U_{\boldsymbol{K}}=\mathscr{B}_{\omega}^{\Rsh}(\boldsymbol{F}_{\min})\setminus\bigcup_{x,y\in\boldsymbol{K}}\boldsymbol{F}_{\min}^{(x,y)_{\Rsh}}.
\end{equation*}
Then we have that $U_{\mathbf{K}}\in {\mathscr B}_{\mathrm{Ac}}(\mathscr{O})$ and the following conditions hold
\begin{equation*}
  U_{\boldsymbol{K}}\cdot \{(i,k,j)\}\subseteq  U_{(i_1,j_1),\ldots,(i_n,j_n)},
\end{equation*}
\begin{equation*}
  \{(i,k,j)\}\cdot U_{\boldsymbol{K}}\subseteq  U_{(i_1,j_1),\ldots,(i_n,j_n)},
\end{equation*}
\begin{equation*}
  \left\{\mathscr{O}\right\}\cdot  \{(i,k,j)\}= \{(i,k,j)\}\cdot \left\{\mathscr{O}\right\}=\left\{\mathscr{O}\right\}\subseteq  U_{(i_1,j_1),\ldots,(i_n,j_n)},
\end{equation*}
\begin{equation*}
  \left\{\mathscr{O}\right\}\cdot  U_{(i_1,j_1),\ldots,(i_n,j_n)}= U_{(i_1,j_1),\ldots,(i_n,j_n)}\cdot \left\{\mathscr{O}\right\}=\left\{\mathscr{O}\right\}\subseteq  U_{(i_1,j_1),\ldots,(i_n,j_n)},
\end{equation*}
\begin{equation*}
  \{(i,k,j)\}\cdot\{(l,m,p)\}=\left\{\mathscr{O}\right\}\subseteq  U_{(i_1,j_1),\ldots,(i_n,j_n)}, \quad \hbox{if} \quad j\neq l,
\end{equation*}
\begin{equation*}
  \{(i,k,j)\}\cdot\{(l,m,p)\}=\{(i,\min\{k,m\},p)\}, \quad \hbox{if} \quad j=l,
\end{equation*}
\begin{equation*}
  \left(U_{(j_1,i_1),\ldots,(j_n,i_n)}\right)^{-1}\subseteq  U_{(i_1,j_1),\ldots,(i_n,j_n)}
\end{equation*}
Therefore, $\left(\mathscr{B}_{\omega}^{\Rsh}(\boldsymbol{F}_{\min}),\tau_{\mathrm{Ac}}\right)$ is a semitopological inverse semigroup with continuous inversion.
\end{proof}

We recall that a topological space $X$ is said to be
\begin{itemize}
  %\item \emph{quasiregular} if for any non-empty open set $U\subset X$ there exists a non-empty open set $V\subset U$ such that $\operatorname{cl}_X(V) \subseteq U$;
  %\item \emph{semiregular} if $X$ has a base consisting of regular open subsets;
  \item \emph{perfectly normal} if $X$ is normal and and every closed subset of $X$ is a $G_\delta$-set;
  \item \emph{scattered} if $X$ does not contain a non-empty dense-in-itself subspace;
  \item \emph{hereditarily disconnected} (or \emph{totally disconnected}) if $X$ does not contain any connected subsets of cardinality larger than one;
  \item \emph{compact} if each open cover of $X$ has a finite subcover;
  %\item \emph{sequentially compact} if each sequence $\{x_n\}_{n\in\mathbb{N}}$ of $X$ has a convergent subsequence in $X$;
  \item \emph{countably compact} if each open countable cover of $X$ has a finite subcover;
  \item \emph{$H$-closed} if $X$ is a closed subspace of every Hausdorff topological space containing $X$;
  \item \emph{infra H-closed} provided that any continuous image of $X$ into any first countable Hausdorff space is closed (see \cite{Hajek-Todd-1975});
  \item \emph{feebly compact}  if each locally finite open cover of $X$ is finite~\cite{Bagley-Connell-McKnight-Jr-1958};
  \item $d$-\emph{feebly compact} (or \emph{\textsf{DFCC}}) if every discrete family of open subsets in $X$ is finite (see \cite{Matveev-1998});
  \item \emph{pseudocompact} if $X$ is Tychonoff and each continuous real-valued function on $X$ is bounded;
  \item $Y$-\emph{compact} for some topological space $Y$, if the image $f(X)$ is compact for any continuous map $f\colon X\to Y$.
\end{itemize}

The relations between above defined compact-like spaces are presented at the diagram in \cite{Gutik-Sobol=2018}.

\begin{lemma}\label{lemma-3.4}
Every shift-continuous $T_1$-topology $\tau$ on the semigroup $\mathscr{B}_{\omega}^{\Rsh}(\boldsymbol{F}_{\min})$ is regular.
\end{lemma}

\begin{proof}[\textsl{Proof}]
By Proposition~\ref{proposition-3.3} every non-zero element of the semigroup $\mathscr{B}_{\omega}^{\Rsh}(\boldsymbol{F}_{\min})$ is an isolated point in the space $\left(\mathscr{B}_{\omega}^{\Rsh}(\boldsymbol{F}_{\min}),\tau\right)$. Hence every open neighbourhood $V(\mathscr{O})$ of the zero $\mathscr{O}$ is a closed subset in $\left(\mathscr{B}_{\omega}^{\Rsh}(\boldsymbol{F}_{\min}),\tau\right)$, which implies that  the topological space $\left(\mathscr{B}_{\omega}^{\Rsh}(\boldsymbol{F}_{\min}),\tau\right)$ is regular.
\end{proof}

Since in any countable $T_1$-space $X$ every open subset of $X$ is a $F_\sigma$-set, Theorem~1.5.17 from \cite{Engelking-1989} and Lemma~\ref{lemma-3.4} imply the following corollary.

\begin{corollary}\label{corollary-3.5}
Let $\tau$ be a shift-continuous $T_1$-topology on the semigroup $\mathscr{B}_{\omega}^{\Rsh}(\boldsymbol{F}_{\min})$. Then $\left(\mathscr{B}_{\omega}^{\Rsh}(\boldsymbol{F}_{\min}),\tau\right)$  is a perfectly normal, scattered, hereditarily disconnected space.
\end{corollary}

By $\mathfrak{D}(\omega)$  we denote the infinite countable discrete space and by $\mathbb{R}$ the set of all real numbers with the usual topology.

\begin{theorem}\label{theorem-3.6}
Let $\tau$ be a shift-continuous $T_1$-topology on the semigroup $\mathscr{B}_{\omega}^{\Rsh}(\boldsymbol{F}_{\min})$. Then the following statements are equivalent:
\begin{itemize}
  \item[$(i)$] $\left(\mathscr{B}_{\omega}^{\Rsh}(\boldsymbol{F}_{\min}),\tau\right)$ is compact;
  \item[$(ii)$] $\tau=\tau_{\mathrm{Ac}}$;
  \item[$(iii)$] $\left(\mathscr{B}_{\omega}^{\Rsh}(\boldsymbol{F}_{\min}),\tau\right)$ is $H$-closed;
  \item[$(iv)$] $\left(\mathscr{B}_{\omega}^{\Rsh}(\boldsymbol{F}_{\min}),\tau\right)$ is feebly compact;
  \item[$(v)$] $\left(\mathscr{B}_{\omega}^{\Rsh}(\boldsymbol{F}_{\min}),\tau\right)$ is infra $H$-closed;
  \item[$(vi)$] $\left(\mathscr{B}_{\omega}^{\Rsh}(\boldsymbol{F}_{\min}),\tau\right)$ is $d$-feebly compact;
  \item[$(vii)$] $\left(\mathscr{B}_{\omega}^{\Rsh}(\boldsymbol{F}_{\min}),\tau\right)$ is pseudocompact;
  \item[$(viii)$] $\left(\mathscr{B}_{\omega}^{\Rsh}(\boldsymbol{F}_{\min}),\tau\right)$ is $\mathbb{R}$-compact;
  \item[$(ix)$] $\left(\mathscr{B}_{\omega}^{\Rsh}(\boldsymbol{F}_{\min}),\tau\right)$ is $\mathfrak{D}(\omega)$-compact.
\end{itemize}
\end{theorem}

\begin{proof}[\textsl{Proof}]
Implications $(ii)\Rightarrow(i)\Rightarrow(iii)\Rightarrow(iv)\Rightarrow(v)\Rightarrow(viii)\Rightarrow(ix)$ and  $(i)\Rightarrow(vii)\Rightarrow(iv)\Rightarrow(vi)$ are trivial (see the diagram in \cite{Gutik-Sobol=2018}). By Lemma~\ref{lemma-3.4} we get implications $(vi)\Rightarrow(iv)$ and $(iii)\Rightarrow(i)$.

$(ix)\Rightarrow(i)$ Suppose to the contrary that there exists a shift-continuous $T_1$-topology $\tau$ on the semigroup $\mathscr{B}_{\omega}^{\Rsh}(\boldsymbol{F}_{\min})$ such that $\left(\mathscr{B}_{\omega}^{\Rsh}(\boldsymbol{F}_{\min}),\tau\right)$ is a $\mathfrak{D}(\omega)$-compact non-compact space. Then there exists an open cover $\mathscr{U}=\{U_\alpha\}$ of $\left(\mathscr{B}_{\omega}^{\Rsh}(\boldsymbol{F}_{\min}),\tau\right)$ which does not contain a finite subcover. Fix $U_{\alpha_0}\in \mathscr{U}$ such that $\mathscr{O}\in U_{\alpha_0}$. Since the space $\left(\mathscr{B}_{\omega}^{\Rsh}(\boldsymbol{F}_{\min}),\tau\right)$ is not compact the set $\mathscr{B}_{\omega}^{\Rsh}(\boldsymbol{F}_{\min})\setminus U_{\alpha_0}$ is infinite. We enumerate the set $\mathscr{B}_{\omega}^{\Rsh}(\boldsymbol{F}_{\min})\setminus U_{\alpha_0}$, i.e., put $\{\boldsymbol{x}_i\colon i\in\omega\}=\mathscr{B}_{\omega}^{\Rsh}(\boldsymbol{F}_{\min})\setminus U_{\alpha_0}$. We identify $\mathfrak{D}(\omega)$ with $\omega$ and define a map $\mathfrak{f}\colon \left(\mathscr{B}_{\omega}^{\Rsh}(\boldsymbol{F}_{\min}),\tau\right)\to \mathfrak{D}(\omega)$ by the formula
\begin{equation*}
  \mathfrak{f}(\boldsymbol{x})=
\left\{
  \begin{array}{ll}
    0, & \hbox{if~} \boldsymbol{x}\in U_{\alpha_0};\\
    i, & \hbox{if~} \boldsymbol{x}=\boldsymbol{x}_i.
  \end{array}
\right.
\end{equation*}
Proposition~\ref{proposition-3.1} implies that such defined map $\mathfrak{f}$ is continuous. Also, the image $\mathfrak{f}(\mathscr{B}_{\omega}^{\Rsh}(\boldsymbol{F}_{\min}))$ is not a compact subset of $\mathfrak{D}(\omega)$, which contradicts the assumption.
\end{proof}

\begin{remark}\label{remark-3.7}
\begin{enumerate}
  \item By Proposition~4 of \cite{Gutik-Mykhalenych=2020} the semigroup $\boldsymbol{B}_{\omega}^{\mathscr{F}}$ contains an isomorphic copy of the semigroup of $\omega\times\omega$-matrix units. Then Theorem~5 from \cite{Gutik-Pavlyk-Reiter=2009} implies that $\boldsymbol{B}_{\omega}^{\mathscr{F}}$ does not embed into a countably compact Hausdorff topological semigroup.
  \item A Hausdorff topological semigroup $S$ is called \emph{$\Gamma$-compact} if for every $x\in S$ the closure of the set
$\{x,x^2,x^3,\ldots\}$ is compact in $S$ (see \cite{Hildebrant-Koch-1986}). The semigroup operation $\boldsymbol{B}_{\omega}^{\mathscr{F}}$ implies that either $a\cdot a=a$ or $a\cdot a=\mathscr{O}$ for any $a\in \boldsymbol{B}_{\omega}^{\mathscr{F}}$. Hence the semigroup $\boldsymbol{B}_{\omega}^{\mathscr{F}}$ with any Hausdorff semigroup topology is $\Gamma$-compact.
\end{enumerate}
\end{remark}

\section{\textbf{On the closure of $\boldsymbol{B}_{\omega}^{\mathscr{F}}$ in a (semi)topological semigroup}}

\begin{lemma}\label{lemma-4.1}
Let $S$ be a dense subsemigroup of a $T_1$-semitopological semigroup $T$ and $0$ be the zero of $S$. Then the element $0$ is the zero of $T$.
\end{lemma}

\begin{proof}[\textsl{Proof}]
Suppose to the contrary that there exists $a\in T\setminus S$ such that $0\cdot a=b\neq0$. Then for every open neighbourhood $U(b)\not\ni 0$ in $T$ there exists an open neighbourhood $V(a)\not\ni 0$ of the point $a$ in $T$ such that $0\cdot V (a)\subseteq U(b)$. But $|V(a)\cap S|\geqslant \omega$, and hence $0\in 0\cdot V (a)\subseteq U(b)$. This contradicts the choice of the neighbourhood $U(b)$. Therefore $0\cdot a=0$ for all $a\in T\setminus S$.

The proof of the equality $a\cdot 0=0$ is similar.
\end{proof}

\begin{theorem}\label{theorem-4.2}
Let $T$ be a $T_1$-semitopological semigroup which contains the semigroup $\boldsymbol{B}_{\omega}^{\mathscr{F}}$ as a dense proper subsemigroup. Then $I=\Big(T\setminus \boldsymbol{B}_{\omega}^{\mathscr{F}}\Big)\cup\{\boldsymbol{0}\}$ is an ideal of $T$.
\end{theorem}

\begin{proof}[\textsl{Proof}]
Lemma~\ref{lemma-4.1} implies that $\boldsymbol{0}$ is the zero of the semigroup $T$. Since $T$ is a  $T_1$-topological space, the set $\boldsymbol{B}_{\omega}^{\mathscr{F}}\setminus\{\boldsymbol{0}\}$ is dense in $T$.
By Lemma~3 \cite{Gutik-Savchuk=2017}, $\boldsymbol{B}_{\omega}^{\mathscr{F}}\setminus\{\boldsymbol{0}\}$ is an open subspace of  $T$.

Fix an arbitrary non-zero element $y\in I$. If $x\cdot y=z\notin I$ for some $x\in \boldsymbol{B}_{\omega}^{\mathscr{F}}\setminus\{\boldsymbol{0}\}$ then there exists an open neighbourhood $U(y)$ of the point $y$ in the space $T$ such that
\begin{equation*}
\{x\}\cdot U(y)=\{z\}\subset \boldsymbol{B}_{\omega}^{\mathscr{F}}\setminus\{\boldsymbol{0}\}.
\end{equation*}
By Lemma~\ref{lemma-2.7} the open neighbourhood $U(y)$ should contain finitely many elements of the set $\boldsymbol{B}_{\omega}^{\mathscr{F}}\setminus\{\boldsymbol{0}\}$ which contradicts our assumption. Hence $x\cdot y\in I$ for all $x\in \boldsymbol{B}_{\omega}^{\mathscr{F}}\setminus\{\boldsymbol{0}\}$ and $y\in I$. The proof of the statement that $y\cdot x\in I$ for all $x\in \boldsymbol{B}_{\omega}^{\mathscr{F}}\setminus\{\boldsymbol{0}\}$ and $y\in I$ is similar.

Suppose to the contrary that $x\cdot y=w\notin I$ for some non-zero elements $x,y\in I$. Then $w\in \boldsymbol{B}_{\omega}^{\mathscr{F}}\setminus\{\boldsymbol{0}\}$ and the separate continuity of the semigroup operation in $T$ yields open neighbourhoods $U(x)$ and $U(y)$ of the points $x$ and $y$ in the space $T$, respectively, such that $\{x\}\cdot U(y)=\{w\}$ and $U(x)\cdot \{y\}=\{w\}$. Since both neighbourhoods $U(x)$ and $U(y)$ contain infinitely many elements of the set $\boldsymbol{B}_{\omega}^{\mathscr{F}}\setminus\{\boldsymbol{0}\}$,  equalities $\{x\}\cdot U(y)=\{w\}$ and $U(x)\cdot \{y\}=\{w\}$ do not hold, because $\{x\}\cdot \left(U(y)\cap \boldsymbol{B}_{\omega}^{\mathscr{F}}\setminus\{\boldsymbol{0}\}\right)\subseteq I$. The obtained contradiction implies that $x\cdot y\in I$.
\end{proof}

A subset $D$ of a semigroup $S$ is said to be $\omega$-unstable if $D$ is infinite and $aB\cup Ba\nsubseteq D$ for any $a\in D$ and any infinite subset $B\subseteq D$.

\begin{definition}[{\!\!\cite{Gutik-Lawson-Repov=2009}}]\label{definition-4.3}
An \emph{ideal series} (see, for example, \cite{Clifford-Preston-1961, Clifford-Preston-1967}) for a semigroup $S$ is a chain of ideals
\begin{equation*}
I_0\subseteq I_1 \subseteq I_2 \subseteq\cdots\subseteq I_n = S.
\end{equation*}
We call the ideal series \emph{tight} if $I_0$ is a finite set and $D_k=I_k\setminus I_{k-1}$ is an $\omega$-unstable subset for each $k=1,\ldots,n$.
\end{definition}

\begin{lemma}\label{lemma-4.4}
The ideal series $I_0=\{\mathscr{O}\}\subset I_1=\mathscr{B}_{\omega}^{\Rsh}(\boldsymbol{F}_{\min})$ is tight for the semigroup $\mathscr{B}_{\omega}^{\Rsh}(\boldsymbol{F}_{\min})$.
\end{lemma}

\begin{proof}[\textsl{Proof}]
Fix any infinite subset $D\subseteq \mathscr{B}_{\omega}^{\Rsh}(\boldsymbol{F}_{\min})\setminus \{\mathscr{O}\}$ and any element $a\in \mathscr{B}_{\omega}^{\Rsh}(\boldsymbol{F}_{\min})\setminus \{\mathscr{O}\}$. Since the set $D$ is infinite and the set $\boldsymbol{F}_{\min}^{(i,j)_{\Rsh}}$ is finite for any $i,j\in\omega$, at least one of the following conditions holds:
\begin{itemize}
  \item[$(i)$] there exist infinitely many $i_n\in\omega$ such that $(i_n,k_n,j_n)\in D$ for some $j_n\in \omega$ and $k_n\in \boldsymbol{F}_{\min}$;
  \item[$(ii)$] there exist infinitely many $j_n\in\omega$ such that $(i_n,k_n,j_n)\in D$ for some $i_n\in \omega$ and $k_n\in \boldsymbol{F}_{\min}$.
\end{itemize}
Both above conditions and the semigroup operation of $\mathscr{B}_{\omega}^{\Rsh}(\boldsymbol{F}_{\min})$ imply that $\mathscr{O}\in(i,k,j)\cdot D\cup D\cdot (i,k,j)$, which completes the proof of the lemma.
\end{proof}

Let $\mathfrak{S}$ be a class of semitopological semigroups. A semigroup $S\in\mathfrak{S}$ is called {\it $\mathfrak{S}$-closed}, if $S$ is a closed subsemigroup of any semitopological semigroup $T\in\mathfrak{S}$ which contains $S$ both as a subsemigroup and as a topological space. $\mathscr{H\!T\!S}$-closed topological semigroups, where $\mathscr{H\!T\!S}$ is the class of  Hausdorff topological semigroups, are introduced by Stepp in \cite{Stepp=1969}, and there they were called {\it maximal semigroups}.  An algebraic semigroup $S$ is called {\it algebraically complete in} $\mathfrak{S}$, if $S$ with any Hausdorff topology $\tau$ such that $(S,\tau)\in\mathfrak{S}$ is $\mathfrak{S}$-closed.

By Proposition~10 from \cite{Gutik-Lawson-Repov=2009}, every inverse semigroup $S$ with a tight ideal series is algebraically complete in the class of Hausdorff semitopological inverse semigroups with continuous inversion. Hence Theorem~\ref{theorem-2.4} and Lemma~\ref{lemma-4.4} imply the following theorem.

\begin{theorem}\label{theorem-4.5}
Let $\mathscr{F}$ be a family of atomic subsets of $\omega$. Then the semigroup $\boldsymbol{B}_{\omega}^{\mathscr{F}}$ is algebraically complete in the class of Hausdorff semitopological inverse semigroups with continuous inversion.
\end{theorem}

The following lemma describes the closure of the semigroup $\mathscr{B}_{\omega}^{\Rsh}(\boldsymbol{F}_{\min})$ in a $T_1$-topological semigroup.

\begin{lemma}\label{lemma-4.6}
Let $S$ be a $T_1$-topological semigroup which contains the semigroup $\mathscr{B}_{\omega}^{\Rsh}(\boldsymbol{F}_{\min})$ as a dense subsemigroup. Then the following conditions hold:
\begin{itemize}
  \item[$(i)$] if $S\setminus \mathscr{B}_{\omega}^{\Rsh}(\boldsymbol{F}_{\min})\neq\varnothing$ then $x^2=\mathscr{O}$ for all $x\in S\setminus \mathscr{B}_{\omega}^{\Rsh}(\boldsymbol{F}_{\min})$;
  \item[$(ii)$] $E(S)=E(\mathscr{B}_{\omega}^{\Rsh}(\boldsymbol{F}_{\min}))$.
\end{itemize}
\end{lemma}

\begin{proof}[\textsl{Proof}]
$(i)$ By Lemma~\ref{lemma-4.1} the element $\mathscr{O}$ is the zero of the semigroup $S$. Suppose to the contrary that there exists $x\in S\setminus \mathscr{B}_{\omega}^{\Rsh}(\boldsymbol{F}_{\min})$ such that $x^2=y\neq\mathscr{O}$. Since $S$ is a $T_1$-space there exists an open neighbourhood $U(y)$ of the point $y$ in $S$ such that $\mathscr{O}\notin U(y)$. The continuity of the semigroup operation in $S$ implies that there exists an open neighbourhood $V(x)$ of the point $x$ in the space $S$ such that $V(x)\cdot V(x)\subseteq U(y)$. By Corollary~\ref{corollary-2.2} the set $\boldsymbol{F}_{\min}^{(i,j)_{\Rsh}}$ is finite for any $i,j\in\omega$. Since the set $V(x)\cap \mathscr{B}_{\omega}^{\Rsh}(\boldsymbol{F}_{\min})$ is infinite, the above arguments and the definition of the semigroup operation in $\mathscr{B}_{\omega}^{\Rsh}(\boldsymbol{F}_{\min})$ imply that $\mathscr{O}\in V(x)\cdot V(x)\subseteq U(y)$, a contradiction.

\smallskip

Statement $(ii)$ follows from $(i)$.
\end{proof}

\begin{lemma}\label{lemma-4.7}
Let $\mathscr{B}_{\omega}^{\Rsh}(\boldsymbol{F}_{\min})$ be a Hausdorff topological semigroup with the compact band $E(\mathscr{B}_{\omega}^{\Rsh}(\boldsymbol{F}_{\min}))$. If a Hausdorff topological semigroup $S$ contains $\mathscr{B}_{\omega}^{\Rsh}(\boldsymbol{F}_{\min})$ as a subsemigroup then $\mathscr{B}_{\omega}^{\Rsh}(\boldsymbol{F}_{\min})$ is a closed subset of $S$.
\end{lemma}

\begin{proof}[\textsl{Proof}]
Suppose to the contrary that there exists a Hausdorff topological semigroup $S$ which contains $\mathscr{B}_{\omega}^{\Rsh}(\boldsymbol{F}_{\min})$ as a non-closed subsemigroup. Since the closure of a subsemigroup of $S$ is again a subsemigroup in $S$ (see \cite[page~9]{Carruth-Hildebrant-Koch-1983}), without loss of generality we may assume that $\mathscr{B}_{\omega}^{\Rsh}(\boldsymbol{F}_{\min})$ is a dense subsemigroup of $S$ and $S\setminus \mathscr{B}_{\omega}^{\Rsh}(\boldsymbol{F}_{\min})\neq\varnothing$. By Lemma~\ref{lemma-4.1} the element $\mathscr{O}$ is the zero of $S$.

Fix an arbitrary $x\in S\setminus \mathscr{B}_{\omega}^{\Rsh}(\boldsymbol{F}_{\min})$. By Hausdorffness of $S$ there exist open neighbourhoods $U(x)$ and $U(\mathscr{O})$ of the points $x$ and $\mathscr{O}$ in $S$, respectively, such that $U(x)\cap U(\mathscr{O})=\varnothing$. Since $x\cdot \mathscr{O}=\mathscr{O}\cdot x=\mathscr{O}$, there exist open neighbourhoods $V(x)$ and $V(\mathscr{O})$ of the points $x$ and $\mathscr{O}$ in the space $S$, respectively, such that
\begin{equation*}
  V(x)\cdot V(\mathscr{O})\subseteq U(\mathscr{O}), \qquad V(\mathscr{O})\cdot V(x)\subseteq U(\mathscr{O}), \qquad V(x)\subseteq U(x) \quad \hbox{and} \quad V(\mathscr{O})\subseteq U(\mathscr{O}).
\end{equation*}
The compactness of $E(\mathscr{B}_{\omega}^{\Rsh}(\boldsymbol{F}_{\min}))$ and Proposition~\ref{proposition-3.1} imply that the set $E(\mathscr{B}_{\omega}^{\Rsh}(\boldsymbol{F}_{\min}))\setminus V(\mathscr{O})$ is finite. Also, by Corollary~\ref{corollary-2.2} the set $\boldsymbol{F}_{\min}^{(i,j)_{\Rsh}}$ is finite for any $i,j\in\omega$. Since the set $V(x)\cap \mathscr{B}_{\omega}^{\Rsh}(\boldsymbol{F}_{\min})$ is infinite, the above arguments and the definition of the semigroup operation in $\mathscr{B}_{\omega}^{\Rsh}(\boldsymbol{F}_{\min})$ imply that there exists $(i,k,j)\in V(x)$ such that  $(i,k,i)\in V(\mathscr{O})$ or $(j,k,j)\in V(\mathscr{O})$. Therefore, we have that at least one of the following conditions holds:
\begin{equation*}
  (V(x)\cdot V(\mathscr{O}))\cap V(x)\neq\varnothing, \qquad (V(\mathscr{O})\cdot V(x))\cap V(x)\neq\varnothing.
\end{equation*}
Since $V(x)\subseteq U(x)$, this contradicts the assumption $U(x)\cap U(\mathscr{O})=\varnothing$. The obtained contradiction implies
the statement of the lemma.
\end{proof}

Later by $\mathscr{H\!T\!S}$ we denote the class of all Hausdorff topological semigroups.

The following lemma shows that the converse statement to Lemma~\ref{lemma-4.7} is true in the case when $\mathscr{B}_{\omega}^{\Rsh}(\boldsymbol{F}_{\min})$ is a topological inverse semigroup.

\begin{lemma}\label{lemma-4.8}
Let $(\mathscr{B}_{\omega}^{\Rsh}(\boldsymbol{F}_{\min}),\tau)$ be a Hausdorff topological inverse semigroup. If \linebreak $(\mathscr{B}_{\omega}^{\Rsh}(\boldsymbol{F}_{\min}),\tau)$ is an $\mathscr{H\!T\!S}$-closed topological semigroup then the band $E(\mathscr{B}_{\omega}^{\Rsh}(\boldsymbol{F}_{\min}))$ is compact.
\end{lemma}

\begin{proof}[\textsl{Proof}]
Suppose to the contrary that there exists a Hausdorff semigroup inverse topology $\tau$ on the semigroup $\mathscr{B}_{\omega}^{\Rsh}(\boldsymbol{F}_{\min})$ such that $(\mathscr{B}_{\omega}^{\Rsh}(\boldsymbol{F}_{\min}),\tau)$ is an $\mathscr{H\!T\!S}$-closed topological semigroup and the band $E(\mathscr{B}_{\omega}^{\Rsh}(\boldsymbol{F}_{\min}))$ is not compact. By Proposition~\ref{proposition-3.1} every non-zero element of $\mathscr{B}_{\omega}^{\Rsh}(\boldsymbol{F}_{\min})$ is an isolated point in $\left(\mathscr{B}_{\omega}^{\Rsh}(\boldsymbol{F}_{\min}),\tau\right)$ and hence there exists an open neighbourhood $V(\mathscr{O})$ of the zero $\mathscr{O}$ in the space $\left(\mathscr{B}_{\omega}^{\Rsh}(\boldsymbol{F}_{\min}),\tau\right)$ such that $M=E(\mathscr{B}_{\omega}^{\Rsh}(\boldsymbol{F}_{\min}))\setminus V(\mathscr{O})$ is an infinite subset of the band $E(\mathscr{B}_{\omega}^{\Rsh}(\boldsymbol{F}_{\min}))$. Since the semigroup $\mathscr{B}_{\omega}^{\Rsh}(\boldsymbol{F}_{\min})$ is countable, so is the set $M$. Next we enumerate elements of the set $M$ by positive integers:
\begin{equation*}
  M=\{(i_n,k_n,i_n)\colon n=1,2,3,\ldots\}.
\end{equation*}
By Corollary~\ref{corollary-2.2} the set $\boldsymbol{F}_{\min}^{(i,j)_{\Rsh}}$ is finite for any $i,j\in\omega$, and hence without loss of generality we may assume that $i_m< i_n$ for any positive integers $m<n$. Since $(\mathscr{B}_{\omega}^{\Rsh}(\boldsymbol{F}_{\min}),\tau)$ is a  topological inverse semigroup the maps $\varphi\colon \mathscr{B}_{\omega}^{\Rsh}(\boldsymbol{F}_{\min})\to E(\mathscr{B}_{\omega}^{\Rsh}(\boldsymbol{F}_{\min}))$ and $\psi\colon \mathscr{B}_{\omega}^{\Rsh}(\boldsymbol{F}_{\min})\to E(\mathscr{B}_{\omega}^{\Rsh}(\boldsymbol{F}_{\min}))$ defined by the formulae $\varphi(\boldsymbol{x})=\boldsymbol{x}\cdot\boldsymbol{x}^{-1}$ and $\psi(\boldsymbol{x})=\boldsymbol{x}^{-1}\cdot \boldsymbol{x}$, respectively, are continuous, and hence  $\mathcal{I}_M=\varphi^{-1}(M)\cup\psi^{-1}(M)$ is a closed subset in the topological space $(\mathscr{B}_{\omega}^{\Rsh}(\boldsymbol{F}_{\min}),\tau)$.

Let $\boldsymbol{y}\notin \mathscr{B}_{\omega}^{\Rsh}(\boldsymbol{F}_{\min})$. Put $S=\mathscr{B}_{\omega}^{\Rsh}(\boldsymbol{F}_{\min})\cup \{\boldsymbol{y}\}$. We extend the semigroup operation from $\mathscr{B}_{\omega}^{\Rsh}(\boldsymbol{F}_{\min})$ onto $S$ as follows:
\begin{equation*}
  \boldsymbol{y}\cdot \boldsymbol{y}=\boldsymbol{y}\cdot\boldsymbol{x}=\boldsymbol{x}\cdot\boldsymbol{y}=\mathscr{O}, \qquad \hbox{for all} \quad \boldsymbol{x}\in \mathscr{B}_{\omega}^{\Rsh}(\boldsymbol{F}_{\min}).
\end{equation*}
Simple verifications show that so extended binary operation is associative.

We put
\begin{equation*}
M_n=\left\{(i_{2j-1},k_{2j-1}, i_{2j})\colon j=n,n+1,n+2,\ldots\right\}
\end{equation*}
for any positive integer $n$. We define a topology $\tau_S$ on $S$ in the following way:
\begin{itemize}
  \item[$(i)$] for every $\boldsymbol{x}\in \mathscr{B}_{\omega}^{\Rsh}(\boldsymbol{F}_{\min})$ the bases of topologies $\tau$ and $\tau_S$ at the point $\boldsymbol{x}$ coincide; \quad and
  \item[$(ii)$] the family $\mathscr{B}=\left\{U_n(\boldsymbol{y})=\{\boldsymbol{y}\}\cup M_n\colon n=1,2,3,\ldots\right\}$ is the base of the topology $\tau_S$ at the point $\boldsymbol{y}$.
\end{itemize}
Since $M_n\subset \mathcal{I}_M$ for any positive integer $n$, $\tau_S$ is a Hausdorff topology on $S$.

For any open neighbourhood $V(\mathscr{O})$ of the zero $\mathscr{O}$ such that $V(\mathscr{O})\subseteq U(\mathscr{O})$ and any positive integer $n$ we have that
\begin{equation*}
  V(\mathscr{O})\cdot U_n(\boldsymbol{y})=U_n(\boldsymbol{y})\cdot V(\mathscr{O})=U_n(\boldsymbol{y})\cdot U_n(\boldsymbol{y})=\{\mathscr{O}\}\subseteq V(\mathscr{O}).
\end{equation*}
We remark that the definition of the set $M_n$ implies that for any non-zero element $(i,k,j)$ of the semigroup $\mathscr{B}_{\omega}^{\Rsh}(\boldsymbol{F}_{\min})$ there exists the smallest positive integer $n_{(i,k,j)}$ such that
\begin{equation*}
  (i,k,j)\cdot M_{n_{(i,k,j)}}=M_{n_{(i,k,j)}}\cdot (i,k,j)=\{\mathscr{O}\}.
\end{equation*}
This implies that
\begin{equation*}
  (i,k,j)\cdot U_{n_{(i,k,j)}}(\boldsymbol{y})=U_{n_{(i,k,j)}}(\boldsymbol{y})\cdot (i,k,j)=\{\mathscr{O}\}\subseteq V(\mathscr{O}).
\end{equation*}
Therefore $(S,\tau_S)$ is a Hausdorff topological semigroup which contains $(\mathscr{B}_{\omega}^{\Rsh}(\boldsymbol{F}_{\min}),\tau)$ as a proper dense subsemigroup, which contradicts the assumption of the lemma. The obtained contradiction implies that the band $E(\mathscr{B}_{\omega}^{\Rsh}(\boldsymbol{F}_{\min}))$ is compact.
\end{proof}

The proof of Lemma~\ref{lemma-4.8} implies Proposition~\ref{proposition-4.9}, which gives the sufficient conditions on the topological semigroup $(\mathscr{B}_{\omega}^{\Rsh}(\boldsymbol{F}_{\min}),\tau)$ to be non-$\mathscr{H\!T\!S}$-closed.

\begin{proposition}\label{proposition-4.9}
Let $\tau$ be a semigroup topology on the semigroup $\mathscr{B}_{\omega}^{\Rsh}(\boldsymbol{F}_{\min})$. Let $\varphi\colon \mathscr{B}_{\omega}^{\Rsh}(\boldsymbol{F}_{\min})\to E(\mathscr{B}_{\omega}^{\Rsh}(\boldsymbol{F}_{\min}))$ and $\psi\colon \mathscr{B}_{\omega}^{\Rsh}(\boldsymbol{F}_{\min})\to E(\mathscr{B}_{\omega}^{\Rsh}(\boldsymbol{F}_{\min}))$ be the maps which are defined by the formulae $\varphi(\boldsymbol{x})=\boldsymbol{x}\cdot\boldsymbol{x}^{-1}$ and $\psi(\boldsymbol{x})=\boldsymbol{x}^{-1}\cdot \boldsymbol{x}$. If there exists an open neighbourhood $U(\mathscr{O})$ of zero in $(\mathscr{B}_{\omega}^{\Rsh}(\boldsymbol{F}_{\min}),\tau)$ such that
\begin{equation*}
\left(\varphi^{-1}(M)\cup\psi^{-1}(M)\right)\cap U(\mathscr{O})=\varnothing
\end{equation*}
for some infinite subset $M$ of the band $E(\mathscr{B}_{\omega}^{\Rsh}(\boldsymbol{F}_{\min}))$, then $(\mathscr{B}_{\omega}^{\Rsh}(\boldsymbol{F}_{\min}),\tau)$ is not an $\mathscr{H\!T\!S}$-closed topological semigroup.
\end{proposition}

Theorem~\ref{theorem-2.4} and Lemmas~\ref{lemma-4.7}, \ref{lemma-4.8} imply

\begin{theorem}\label{theorem-4.10}
Let $\mathscr{F}$ be a some family of atomic subsets of $\omega$. Then a Hausdorff topological semigroup $\boldsymbol{B}_{\omega}^{\mathscr{F}}$ with the compact band is an $\mathscr{H\!T\!S}$-closed topological semigroup. Moreover, a Hausdorff topological inverse semigroup $\boldsymbol{B}_{\omega}^{\mathscr{F}}$ is an $\mathscr{H\!T\!S}$-closed topological semigroup if and only the band $E(\boldsymbol{B}_{\omega}^{\mathscr{F}})$ is compact.
\end{theorem}

Example~\ref{example-4.11} and Proposition~\ref{proposition-4.12} imply that the converse statement to Lemma~\ref{lemma-4.7} (and hence to the first statement of Theorem~\ref{theorem-2.4}) is not true.

\begin{example}\label{example-4.11}
For any positive integer $n$ we denote
\begin{equation*}
  U_n(\mathscr{O})=\{\mathscr{O}\}\cup\bigcup\left\{\boldsymbol{F}_{\min}^{(i,j)_{\Rsh}}\colon n\leqslant i<j\right\}.
\end{equation*}
We define a topology $\tau_1$ on the semigroup $\mathscr{B}_{\omega}^{\Rsh}(\boldsymbol{F}_{\min})$ in the following way:
\begin{itemize}
  \item[$(i)$] any non-zero element of the semigroup $\mathscr{B}_{\omega}^{\Rsh}(\boldsymbol{F}_{\min})$ is an isolated point in $(\mathscr{B}_{\omega}^{\Rsh}(\boldsymbol{F}_{\min}),\tau_1)$;
  \item[$(ii)$] the family $\mathscr{B}_1(\mathscr{O})=\left\{U_n(\mathscr{O})\colon n\in\omega\right\}$ is the base of the topology $\tau_1$ at the zero  $\mathscr{O}$.
\end{itemize}
It is obvious that $(\mathscr{B}_{\omega}^{\Rsh}(\boldsymbol{F}_{\min}),\tau_1)$ is a Hausdorff topological space.
\end{example}

\begin{proposition}\label{proposition-4.12}
$(\mathscr{B}_{\omega}^{\Rsh}(\boldsymbol{F}_{\min}),\tau_1)$ is  an $\mathscr{H\!T\!S}$-closed topological semigroup.
\end{proposition}

\begin{proof}[\textsl{Proof}]
First we show that the semigroup operation is continuous in $(\mathscr{B}_{\omega}^{\Rsh}(\boldsymbol{F}_{\min}),\tau_1)$. Since every non-zero element of the semigroup $(\mathscr{B}_{\omega}^{\Rsh}(\boldsymbol{F}_{\min}),\tau_1)$ is an isolated point, it is complete
to show that the semigroup operation in $(\mathscr{B}_{\omega}^{\Rsh}(\boldsymbol{F}_{\min}),\tau_1)$ is continuous at zero. Fix an arbitrary $(i,k,j)\in \mathscr{B}_{\omega}^{\Rsh}(\boldsymbol{F}_{\min})\setminus \{\mathscr{O}\}$. Then for $n=\max\{i,j\}+1$ we have that
\begin{equation*}
  (i,k,j)\cdot U_n(\mathscr{O})=U_n(\mathscr{O})\cdot (i,k,j)=\{\mathscr{O}\}\subset U_n(\mathscr{O}).
\end{equation*}
Also for any $n\in\omega$ we have that
\begin{equation*}
  U_n(\mathscr{O})\cdot U_n(\mathscr{O})\subseteq U_n(\mathscr{O}).
\end{equation*}
Therefore $(\mathscr{B}_{\omega}^{\Rsh}(\boldsymbol{F}_{\min}),\tau_1)$ is a topological semigroup.

Suppose to the contrary that there exists a Hausdorff topological semigroup $S$ which contains $(\mathscr{B}_{\omega}^{\Rsh}(\boldsymbol{F}_{\min}),\tau_1)$ as a non-closed subsemigroup. Since the closure of
a subsemigroup in a topological semigroup is a subsemigroup (see \cite[page~9]{Carruth-Hildebrant-Koch-1983}), without loss of
generality we can assume that $(\mathscr{B}_{\omega}^{\Rsh}(\boldsymbol{F}_{\min}),\tau_1)$ is a dense proper subsemigroup of $S$.

Fix an arbitrary $\boldsymbol{x}\in S\setminus\mathscr{B}_{\omega}^{\Rsh}(\boldsymbol{F}_{\min})$. By Lemmas~\ref{lemma-4.1} and \ref{lemma-4.6} we have that
\begin{equation*}
\boldsymbol{x}\cdot \boldsymbol{x}=\boldsymbol{x}\cdot \mathscr{O}=\mathscr{O}\cdot\boldsymbol{x}=\mathscr{O}.
\end{equation*}
Fix any positive integer $n$. Let $W(\mathscr{O})$ be an open neighbourhood of zero $\mathscr{O}$ in $S$ such that $W(\mathscr{O})\cap \mathscr{B}_{\omega}^{\Rsh}(\boldsymbol{F}_{\min})= U_n(\mathscr{O})$. The continuity of the semigroup operation in $S$ implies that there exist open neighbourhoods $V(\boldsymbol{x})$, $V(\mathscr{O})$ and $U(\mathscr{O})$ of the points $\boldsymbol{x}$ and $\mathscr{O}$ in the space $S$, respectively, such that
\begin{equation*}
V(\boldsymbol{x})\cdot V(\mathscr{O})\subseteq U(\mathscr{O}), \quad V(\mathscr{O})\cdot V(\boldsymbol{x})\subseteq U(\mathscr{O}), \quad V(\boldsymbol{x})\cdot V(\boldsymbol{x})\subseteq U(\mathscr{O}),
\end{equation*}
\begin{equation*}
V(\boldsymbol{x})\cap U(\mathscr{O})=\varnothing \quad \hbox{and} \quad V(\mathscr{O})\subseteq U(\mathscr{O})\subseteq W(\mathscr{O}).
\end{equation*}

Theorem 9 of \cite{Stepp=1975} implies that $E(\mathscr{B}_{\omega}^{\Rsh}(\boldsymbol{F}_{\min}))$ is a closed subset of $S$. Hence, we may assume that $V(\boldsymbol{x})\cap E(\mathscr{B}_{\omega}^{\Rsh}(\boldsymbol{F}_{\min}))=\varnothing$, and moreover $U(\mathscr{O})\cap \mathscr{B}_{\omega}^{\Rsh}(\boldsymbol{F}_{\min})= U_m(\mathscr{O})$ and $V(\mathscr{O})\cap \mathscr{B}_{\omega}^{\Rsh}(\boldsymbol{F}_{\min})= U_l(\mathscr{O})$ for some positive integers $l$ and $m$ such that $l\geqslant m\geqslant n$.

Then conditions
\begin{equation*}
V(\boldsymbol{x})\cdot V(\mathscr{O})\subseteq U(\mathscr{O}) \qquad \hbox{and} \qquad V(\boldsymbol{x})\cap U(\mathscr{O})=\varnothing
\end{equation*}
imply that there exists on open neighbourhood $V_1(\boldsymbol{x})\subseteq V(\boldsymbol{x})$ of the point $\boldsymbol{x}$ in the space $S$ such that
\begin{equation*}
  V_1(\boldsymbol{x})\cap\left(\bigcup\left\{\boldsymbol{F}_{\min}^{(i,s)_{\Rsh}}\colon s\in\omega\right\}\right)=\varnothing
\end{equation*}
for any non-negative integer $i<m$. This and Theorem 9 of \cite{Stepp=1975} imply that there exists an open neighbourhood $V_2(\boldsymbol{x})\subseteq V(\boldsymbol{x})$ of the point $\boldsymbol{x}$ in $S$ such that
\begin{equation*}
V_2(\boldsymbol{x})\cap \mathscr{B}_{\omega}^{\Rsh}(\boldsymbol{F}_{\min})\subseteq \bigcup\left\{\boldsymbol{F}_{\min}^{(i,j)_{\Rsh}}\colon i>j, \; i,j\in\omega\right\}.
\end{equation*}
Hence there exists an infinite sequence $\{(i_p,k_p,j_p)\}_{p\in\omega}$ in $V_2(\boldsymbol{x})$ such that the sequence $\{i_p\}_{p\in\omega}$ is increasing and $j_p\leqslant i_p-1$ for any $p\in\omega$. The definition of the topology $\tau_1$ implies that there exists an element $(i_{p_0},k_{p_0},j_{p_0})$ of the sequence $\{(i_p,k_p,j_p)\}_{p\in\omega}$ such that
\begin{equation*}
\boldsymbol{F}_{\min}^{(i_{p_0}-1,i_{p_0})_{\Rsh}}\subseteq U_l(\mathscr{O})\subseteq V(\mathscr{O}).
\end{equation*}
Then we have that
\begin{equation*}
\boldsymbol{F}_{\min}^{(i_{p_0}-1,i_{p_0})_{\Rsh}}\cdot (i_{p_0},k_{p_0},j_{p_0})\subseteq \boldsymbol{F}_{\min}^{(i_{p_0}-1,j_{p_0})_{\Rsh}} \nsubseteq U_m(\mathscr{O}),
\end{equation*}
which  contradicts the inclusion $ V(\mathscr{O})\cdot V(\boldsymbol{x})\subseteq U(\mathscr{O})$. The obtained contradiction implies that $\boldsymbol{x}$ is not an accumulation point of $\mathscr{B}_{\omega}^{\Rsh}(\boldsymbol{F}_{\min})$ in the topological space $S$, and hence the statement of the proposition holds.
\end{proof}

%%%%%%%%%%%%%%%%%%%%%%%%%%%%%%%%%%%%%%%%%%%%%%%%%%%%%
\section*{\textbf{Acknowledgements}}

The authors acknowledge the referee for their comments and suggestions.
%%%%%%%%%%%%%%%%%%%%%%%%%%%%%%%%%%%%%%%%%%%%%%%%%%%%%%%%%%%%

\end{document}